\documentclass[a4paper,12pt]{amsart}
\usepackage[cp1251]{inputenc}
\usepackage{amssymb}
\usepackage[russian,english]{babel}
\usepackage{rotating}
\usepackage{longtable,float}
\theoremstyle{plain}
\newtheorem{theorem}{Theorem}
\newtheorem{lemma}{Lemma}
\newtheorem{proposition}{Proposition}
\newtheorem{corollary}{Corollary}

\theoremstyle{definition}
\newtheorem{definition}{Definition}
\newtheorem{remark}{Remark}
\newtheorem{example}{Example}
\newcommand{\KK}{\mathbb{K}}
\newcommand{\SL}{\mathrm{SL}}
\newcommand{\PSL}{\mathrm{PSL}}
\newcommand{\GL}{\mathrm{GL}}

\newcommand{\ord}{\mathrm{ord}}
\newcommand{\rk}{\mathrm{rk\,}}
\newcommand{\QQ}{\mathbb{Q}}
\newcommand{\ZZ}{\mathbb{Z}}
\begin{document}

\title[Elements of finite order in the normalizer of a maximal torus]{Elements of finite order in the normalizer of a~maximal torus of a semisimple group}
\author{Ivan~Arzhantsev}
\address{Faculty of Computer Science, HSE University, Pokrovsky Bulvar 11, Moscow, 109028 Russia}
\email{arjantse@hse.ru}
\author{Alexey~Galt}
\address{School of Mathematical Science, Hebei Key Laboratory of Computational Mathematics and Applications, Hebei Normal University, Shijiazhuang 050024, P. R. China
and Sobolev Institute of Mathematics, Koptyuga 4, Novosibirsk,  630090 Russia}
\email{galt84@gmail.com}
\author{Alexey~Staroletov}
\address{Sobolev Institute of Mathematics, Koptyuga 4, Novosibirsk,  630090 Russia}
\email{staroletov@math.nsc.ru}
\thanks{The research of I.~Arzhantsev was done within the framework of the HSE Fundamental Research Program. 
A.~Galt was supported by the grant of the Natural Science Foundation of Hebei Province (project No. A2023205045). }
\subjclass[2020]{Primary 20G07\ Secondary 05A05, 15A18, 20F55}
\keywords{Weyl group, Coxeter element, elliptic element, algebraic group, maximal torus, normalizer, periodic component}
\begin{abstract}
We prove that the set of elements of a given finite order in the connected component $N_w$ of the normalizer $N_G(T)$ of a maximal torus $T$ of a semisimple group~$G$ is either empty or a disjoint union of finitely many irreducible subvarieties $C_i$. The dimension of each $C_i$ equals the dimension of the subspace of fixed vectors for the action of the element~$w$ of the Weyl group $W$ corresponding to the component~$N_w$. Moreover, each $C_i$ is an orbit of the action of the torus $T$ on the component $N_w$ by conjugation.
\end{abstract} 

\maketitle

%%%%%%%%%%%%%%%%%%%%%%%%%%%%%%%%%%
\section{Introduction}

It is well-known that for an arbitrary algebraic group $L$, the irreducible components coincide with the connected components, and the number of such components is finite. The connected component of the identity $L^0$ is a normal subgroup of $L$, and the cosets of $L$ modulo $L^0$ are precisely the connected components of  $L$. The quotient group $\Gamma:=L/L^0$ is called the \emph{component group} of $L$. Thus, the study of an arbitrary algebraic group $L$ often reduces to the study of the connected algebraic group $L^0$ and the finite group $\Gamma$. At the same time, interesting effects are known to arise specifically for disconnected algebraic groups of positive dimension. One such effect is the existence of periodic components.

Let $L$ be a linear algebraic group over an algebraically closed field $\KK$ of characteristic zero. If $L$ has a positive dimension, then the component $L^0$ contains elements of infinite order. However, some components $gL^0$ may consist entirely of elements of finite order. Such components are called \emph{periodic}. For example, the orthogonal group $\operatorname{O}_2(\KK)$ contains exactly two components: the component of the identity consists of matrices with determinant $1$, while the second component consists of matrices with determinant $-1$, which are reflections in the plane $\KK^2$, i.e., they have order $2$.

It is known that if $L$ contains a periodic component $gL^0$, then $L^0$ is solvable; see~\cite[Corollary~10.12]{St}, \cite[Theorem~2]{Fe}. The orders of elements in a periodic component $gL^0$ coincide and, moreover, the action of $L^0$ on $gL^0$ by conjugation is transitive, see~\cite[Theorem~10.1]{St}, \cite[Theorem~1]{Fe}.

Consider the case where the group $L$ is the normalizer $N_G(T)$ of a maximal torus $T$ of a connected semisimple linear algebraic group $G$. In this case, we assume that the field $\KK$ is algebraically closed without restrictions on the characteristic. The component $L^0$ in this case coincides with the torus $T$, and the component group $N_G(T)/T$ is isomorphic to the Weyl group $W$ of the root system of $G$.

Despite its apparent simplicity, several deep problems are associated with the group $N_G(T)$. For instance, J.~Tits~\cite{Ti} formulated the question of when the group $N_G(T)$ is split, i.e., when there exists a finite subgroup $\mathcal{W}$ in $N_G(T)$ such that $N_G(T)$ is a semidirect product $\mathcal{W}\ltimes T$. The answer was obtained independently in the paper~\cite{AH} and in a series of papers \cite{G1}--\cite{G4}. For example, for simple simply connected linear algebraic groups, splitting occurs only for type $\operatorname{A_n}$ with even $n$ and for $\operatorname{G}_2$. We also note that for compact Lie groups, this problem was solved in the paper \cite[Theorem~2]{CWW}; see also~\cite{GLO-1, GLO-2}.

In the case when the group $N_G(T)$ does not split, it is natural to study lifts of specific elements $w\in W$ to $N_G(T)$; for example, one may search an element of the minimal order in the connected component $N_w$ of the group $N_G(T)$ corresponding to $w$. It is clear that the minimal order is at least the order $\ord(w)$ of $w$ in  $W$. It is known that the minimal order of a lift of $w$ is either $\ord(w)$ or $2\cdot\ord(w)$, see~\cite{Ti} or \cite[Lemma~2.1]{AH}. In particular, every connected component of $N_G(T)$ contains an element of finite order.

Periodic components of the group $N_G(T)$ correspond to elliptic elements of the Weyl group $W$~\cite[Proposition~5]{Fe}, \cite[Theorem~4.3]{AZ}. Recall that an element $w\in W$ is called \emph{elliptic} if $w$ acts on the space of the standard linear representation of $W$ without non-zero fixed vectors. In particular, all Coxeter elements of $W$ are elliptic.

The goal of the present paper is to study the structure of the set of finite-order elements in an arbitrary connected component of the group $N_G(T)$. We prove that the elements of finite order are dense in $N_G(T)$. Theorem~\ref{tmain} states that the set of elements of a given finite order in the component $N_w$ is either empty or a disjoint union of a finite number of irreducible closed subvarieties $C_i$. The dimension of each $C_i$ is equal to the dimension of the space of fixed vectors for the action of $w$ on the space of the standard linear representation of the group $W$. Moreover, each $C_i$ is an orbit of the action of the torus $T$ on $N_w$ by conjugation. As a corollary, we obtain previously known results regarding the periodic components of the group $N_G(T)$.

In the final section, we provide two tables containing data on the number and orders of Coxeter elements and other elliptic elements in Weyl groups, the minimal lifts of these elements to the normalizer of a maximal torus and the number of fixed points of the automorphisms of the maximal torus determined by these elements. We also describe the derived subgroup of the normalizer and the subgroups in the normalizer generated by the components corresponding to Coxeter elements and elliptic elements, respectively. In the first table, the corresponding data are collected for the Weyl groups of classical simple linear groups, and in the second table, for the Weyl groups of exceptional simple linear groups.

%%%%%%%%%%%%%%%%%%%%%%%%%%%%%%%%%
\section{Preliminaries}

Let $\KK$ be an algebraically closed field. We introduce some notation related to homomorphisms between arbitrary tori $T_1$ and $T_2$. A homomorphism $\gamma\colon T_1\to T_2$ defines a linear map $\gamma^*\colon M_2\to M_1$ of the character lattices $M_2$ and $M_1$ of the tori $T_2$ and $T_1$, respectively, given by the formula $(\gamma^*\chi)(t):=\chi(\gamma(t))$ for any $t\in T_1$ and $\chi\in M_2$. A choice of bases in the lattices $M_1$ and $M_2$ determines identifications
$$
T_1\cong (\KK^{\times})^k \quad \text{and}  \quad T_2\cong (\KK^{\times})^m.
$$
Let $C=(c_{ij})$ be the matrix of the linear map $\gamma^*$ in these bases. Then, in coordinates, the homomorphism $\gamma=\gamma_C$ has the form
$$
\gamma_C\colon (\KK^{\times})^k\to(\KK^{\times})^m, \quad (t_1,\ldots,t_k)\mapsto (t_1^{c_{11}}\ldots t_k^{c_{k1}},\ldots,t_1^{c_{1m}}\ldots t_k^{c_{km}}).
$$
The composition of homomorphisms corresponds to the product of matrices, while the product of homomorphisms corresponds to the sum of matrices.

Let $\rk C$ denote the rank of the matrix $C$. It is easy to see that the image of the homomorphism $\gamma_C$ is a subtorus in $T_2$ of dimension $\rk C$, while the kernel of the homomorphism $\gamma_C$ is a finite extension of a subtorus of dimension $k-\rk C$ in $T_1$.

Let $n$ be a positive integer and $p$ be a prime number.
The symbol $(n)_p$ denotes the \emph{$p$-part} of the number $n$, i.e., the greatest power of $p$ dividing $n$. In this context, $(n)_{p'}=n/(n)_p$ is called the \emph{$p'$-part} of $n$ (or the \emph{$p$-free part}).

\begin{definition}
The \emph{multiplicity} of an integer matrix $C=(c_{ij})$ is the number of connected components of the kernel of the homomorphism $\gamma_C$.
\end{definition}

If the characteristic of $\KK$ is zero, the multiplicity of the matrix $C$ is equal to the product of non-zero invariant factors of the matrix $C,$ and in the case where $\KK$ has 
characteristic~$p$, the multiplicity is equal to the $p'$-part of this product. In particular, if the matrix $C$ is square and nonsingular, its multiplicity is equal, up to a sign, to the determinant $\det(C)$ and the $p'$-part of this determinant, respectively. For the zero matrix $C$, the multiplicity is equal to~$1$.

Note that when writing the same map $C$ in a basis of a sublattice of finite index, the multiplicity may change. For example, if we pass from a lattice basis $e_1,e_2$ to a sublattice basis $e_1, ue_2$, the matrix $C=\begin{pmatrix} 0 & 1 \\ 0 & 0 \end{pmatrix}$ transforms into the matrix $\begin{pmatrix} 0 & u \\ 0 & 0 \end{pmatrix}$ and its multiplicity changes from $1$ to $u$.

\medskip

Now let $G$ be a connected semisimple linear algebraic group over an algebraically closed field $\KK$. Consider a maximal torus $T$ of $G$, the normalizer $N_G(T)$ of the torus $T$ in $G$, and the Weyl group $W=N_G(T)/T$. We denote by $N_w$ the connected component of $N_G(T)$ corresponding to an element $w\in W$. Fix an arbitrary element $n_w$ from the component $N_w$. We have $N_w=n_wT=Tn_w$.
Let $r=\dim T$.

Each element $w\in W$ defines an automorphism $\varphi_w$ of the torus $T$, given by the formula
$$
\varphi_w(t)=n_wtn_w^{-1}.
$$
The corresponding representation of the group $W$ in the rational vector space $M_{\QQ} :=M\otimes_{\ZZ}\QQ$ spanned by the character lattice $M$ of the torus $T$ is called the standard linear representation of the group $W$.

By $\sigma(w)$ we denote the dimension of the subspace of $w$-fixed points in the space $M_{\QQ}$. In other words, $\sigma(w)$ is the multiplicity of the eigenvalue $1$ for the linear operator defined by the action of the element $w$ in the space $M_{\QQ}$.

An element $w$ of the Weyl group $W$ is called \emph{elliptic} if $\sigma(w)=0$. This definition is taken from~\cite{Lu}; in~\cite{DW}, these elements are called generalized Coxeter elements. An element $w\in W$ is called a \emph{Coxeter element} if $w$ can be represented as a product of simple reflections $s_1,\ldots,s_r \in W$ in some order. It is known that all Coxeter elements are conjugate and elliptic~\cite[Section~3.16]{Hum}. However, there exist elliptic elements that are not Coxeter elements.

The following remarkable formula holds:
$$
\prod_{i=1}^r (1+(d_i-1)t) = a_0 + a_1t + \ldots + a_rt^r=\sum\limits_{w\in W} t^{l(w)},
$$
where $d_i$ are the degrees of the free generators of the algebra of invariants of the standard linear representation of the group $W$, the coefficient $a_i$ is the number of elements $w\in W$ such that $\sigma(w)=r-i$, and $l(w)$ is the minimal number of reflections whose product is~$w$; see~\cite{ShT}, \cite{So}, \cite[Lemma~2]{CarWeyl}, \cite[Section~3.9]{Hum}. In particular, $a_0=1$, $a_1$ is the number of reflections, and $a_r$ is the number of elliptic elements in $W$. This formula shows that the function $\sigma$ takes all integer values from $0$ to $r$ on the group $W$. It also allows one to calculate the number of elliptic elements for each Weyl group, see~\cite[Section~5]{Fe}. The values of the degrees $d_i$ for each root system can be found in~\cite[Section~3.7]{Hum}.

Denote the order of $w \in W$ by $\ord(w)$. It is known (see~\cite{Ti} or \cite[Lemma~2.1]{AH}) that a lift $n_w$ of $w \in W$ to $N_G(T)$ can be chosen such that the element $n_w^{\ord(w)} \in T$ has order at most two. Let $d(w)$ be the smallest order of elements in $N_w$. From the above, it follows that $d(w)$ is equal to $\ord(w)$ or $2 \cdot \ord(w)$. In~\cite{Za}, it is shown that in the case of a simple adjoint algebraic group $G$, for an arbitrary elliptic element $w \in W$ we have $d(w) = \ord(w)$ except for a few cases for groups $G$ of types $\operatorname{C_n}$ and $\operatorname{F_4}$. All these cases are listed in~\cite[Table~3]{Za}. In the same work, the value of $d(w)$ is calculated for all elliptic elements and simply connected simple linear algebraic groups; see also~\cite[Section~5]{Fe}. Henceforth, we assume that the element $n_w$ has order $d(w)$.

\smallskip

Let $\Phi_w$ be the matrix of the operator $\varphi_w$ in some basis of the lattice $M$. Consider the matrix
$$
B_w: = E + \Phi_w + \ldots + \Phi_w^{d-1}, \quad \text{where} \  d=\ord(w).
$$

\begin{definition}
The \emph{multiplicity} of an element $w \in W$ is the multiplicity of the matrix $B_w$. We denote the multiplicity of $w$ by $m(w)$.
\end{definition}

A straightforward calculation shows that in the case of the group $G=\SL_3$ and $w=(12)$, we obtain $m(w)=1$, whereas for $G=\SL_4$ and $w=(12)$, we have $m(w)=2$. It would be interesting to find a general formula for the number $m(w)$. In the following example, we obtain such a formula in the case of the group $\SL_n$.

\begin{example}
Let $G$ be of type $\operatorname{A_{n-1}}$. In this case, $W=\operatorname{S_n}$ and an element $w\in W$ is parameterized by its cycle type $[k_1,\ldots,k_t]$, where $k_1+\ldots+k_t=n$.
We identify the elements of the group $W$ with permutation matrices, and the roots with the vectors $e_i-e_j$ in the $n$-dimensional Euclidean space $V=\langle e_1,\ldots,e_n\rangle$.
Let $d=\ord(w)$. We will show that
$$
m(w)=d^{t-1}\frac{(k_1,\ldots,k_t)}{k_1\cdots k_t},
$$
where $(k_1,\ldots,k_t)$ denotes the greatest common divisor of the numbers $k_1,\ldots,k_t$.
Indeed, let us assume for simplicity that $w$ corresponds to the permutation
$$
(1\ldots k_1)(k_1+1\ldots k_1+k_2)\ldots(k_1+\ldots+k_{t-1}+1\ldots n).
$$
Denote by $v_i=\sum\limits_{j=k_1+\ldots+k_{i-1}+1}^{k_1+\ldots+k_{i-1}+k_i}e_j$ the invariant sum corresponding to the $i$-th cycle of $w$.
Then $(w^0+w+\ldots+w^{d-1})e_j=\frac{d}{k_i}v_i$ for an index $j$ in the $i$-th cycle, so the nonzero invariant factors on the space $V$ are $\{\frac{d}{k_i}~|~1\leq i\leq t\}$. However, we need to find the factors for the action on the sublattice $\langle e_i-e_{i+1}~|~1\leq i\leq n-1\rangle$.
We have
$$
B_w(e_{k_1}-e_{k_1+1})=\frac{d}{k_1}v_1-\frac{d}{k_2}v_2,\quad \ldots, \quad
B_w(e_{k_{1}+\dots+k_{t-1}}-e_{k_{1}+\dots+k_{t-1}+1})=\frac{d}{k_{t-1}}v_{t-1}-\frac{d}{k_t}v_t,
$$
and the action is zero on the remaining vectors. One could further decompose the images in the basis $e_i-e_{i+1}$, but we will not do so, since the resulting matrix would consist of partial sums of the rows of the obtained matrix. Thus, it suffices to find the product of the nonzero invariant factors of the following matrix (considering only the nonzero rows and columns of the previous matrix):

\[\left(\begin{matrix}
\frac{d}{k_1} & 0 & \ldots & 0 & 0\\
-\frac{d}{k_2} & \frac{d}{k_2} & \ldots & 0 & 0\\
0 & -\frac{d}{k_3} &  \ddots & 0 & 0 \\
\ldots & \ldots &  \ddots & \ddots & \ldots\\
0 & 0 &\ldots &  -\frac{d}{k_{t-1}} & \frac{d}{k_{t-1}} \\
0 & 0 &\ldots &  0 & -\frac{d}{k_{t}} \\
\end{matrix}\right)\]
This matrix has size $t\times{(t-1)}$, and therefore the product of its nonzero invariant factors is equal to the greatest common divisor of all minors of order $t-1$. There are $t$ such minors.
It is easy to see that after deleting the $i$-th row, the minor is equal to $\pm\frac{d^{t-1}}{k_1\ldots \hat{k}_i\ldots k_t}$, and the greatest common divisor of these numbers is precisely $d^{t-1}\frac{(k_1,\ldots,k_t)}{k_1\ldots k_t}$.
\end{example}

%%%%%%%%%%%%%%%%%%%%%%%

\section{Main results}

We use the notation introduced in the previous section. Let us begin with a general observation.

\begin{proposition} \label{fod}
Elements of finite order are dense in the group $N_G(T)$.
\end{proposition}

In the following theorem, we provide a more explicit geometric description of the set of elements of a given finite order. Let $\nu(k,r)$ denote the number of elements of order $k$ in a torus of dimension $r$ over the ground field $\KK$. For example, $\nu(1,r)=1$ for any~$r$ and $\nu(k,0)=0$ for all~$k\ge 2$. Also, $\nu(2,r)=2^r-1$ if the characteristic of $\KK$ is not two. If $\KK$ has characteristic~$p$, then $\nu(k,r)=0$ for an arbitrary~$r$ and any~$k$ divisible by $p$. If $k$ is coprime to the characteristic of $\KK$, then $\nu(k,1)$ is equal to the value of the Euler function $\varphi(k)$. More generally, in this case $\nu(k,r)$ coincides with the number of elements of order $k$ in the group $\mathbb{Z}_k^r$. If we denote by $\psi(d)$ the number of elements of order $d$ in this group, then from the equality $\sum\limits_{d|k}\psi(d)=k^r$ and the M\"obius inversion formula, we find that
$\nu(k,r)=\psi(k)=\sum\limits_{d|k}\mu(\frac{k}{d})d^r$.

\begin{theorem} \label{tmain}
Let $m(w)$ be the multiplicity of an element $w\in W$. Then the set $D_k(w)$ of elements of order $k$ in the component $N_w$ is:
\begin{itemize}
\item[a)]
empty if $d(w)$ does not divide $k$;
\item[b)]
a disjoint union of $m(w)\nu(s,\sigma(w))$ irreducible closed subvarieties $C_i$ if $d(w)=\ord(w)$ and $k=d(w)s$;
\item[c)]
a disjoint union of irreducible closed subvarieties $C_i$ if $d(w)=2\cdot\ord(w)$ and $k=d(w)s$, where the number of such subvarieties is
$m(w)(\nu(s,\sigma(w))+\nu(2s,\sigma(w)))$ for odd $s$ and $m(w)\nu(2s,\sigma(w))$ for even~$s$.
\end{itemize}
Each subvariety $C_i$ is isomorphic to a torus of dimension $r-\sigma(w)$ and is an orbit of the action of $T$ on $N_w$ by conjugation.
\end{theorem}

\noindent Proofs of Proposition~\ref{fod} and Theorem~\ref{tmain} are given in the next section.

\medskip

Let us highlight several special cases of Theorem~\ref{tmain}.

\begin{corollary}
The set of elements of the smallest order in the component $N_w$ is irreducible if and only if $m(w)=1$ and either $d(w)=\ord(w)$, or the element $w$ is elliptic, or the characteristic of $\KK$ is equal to two.
\end{corollary}

\begin{proof}
The set of elements of the smallest order corresponds to the case $s=1$. If $d(w)=\ord(w)$, then the number of components of the variety $D_{d(w)}(w)$ is $m(w)\nu(1,\sigma(w))=m(w)=1$. If $d(w)=2\cdot\ord(w)$, then the number of components is
$$
m(w)(\nu(1,\sigma(w))+\nu(2,\sigma(w)))=m(w)+m(w)\nu(2,\sigma(w)).
$$
This number is equal to $1$ if and only if $m(w)=1$ and $\nu(2,\sigma(w))=0$, which means that either $\sigma(w)=0$ or the characteristic of the field $\KK$ is equal to two.
\end{proof}

\begin{corollary}
The sets of elements of a given finite order in the component $N_w$ are at most finite if and only if $N_w=T$.
\end{corollary}

\begin{proof}
The finiteness condition is equivalent to $r-\sigma(w)=0$, whereas the condition ${\sigma(w)=r}$ means that $w$ acts on the space of the standard representation of $W$ as the identity operator.
\end{proof}

\begin{corollary}
All elements of the component $N_w$ have finite order if and only if $w$ is elliptic. In this case, the action of the torus $T$ on $N_w$ by conjugation is transitive, and all elements in $N_w$ have order $d(w)$.
\end{corollary}

\begin{remark}\label{rem2}
In~\cite[Corollary~2]{Fe}, it is noted that if a linear algebraic group $L$ contains a periodic component $gL^0$, then every element of the subgroup $L^0$ is a commutator of two elements from $L$. Indeed, since the action of the group $L^0$ on $gL^0$ by conjugation is transitive, for any $h\in L^0$ there exists $l\in L^0$ such that $gh=lgl^{-1}$, or $h=g^{-1}lgl^{-1}$. Thus, if a subgroup $S$ of the group $L$ contains a periodic component of $L$, then the derived subgroup of $S$ is determined by the derived subgroup of the finite component group of $S$. The fact that any element of $T$ is a commutator of two elements of $N_G(T)$ was previously obtained in~\cite[Lemma~2.3]{Re}.
\end{remark}

\begin{corollary}
The set of elements of a given finite order in the component $N_w$ is a union of a finite number of curves if and only if $w$ is a reflection.
\end{corollary}

\begin{example} \label{ex1}
Consider the case $G=\SL_3$. Here, the Weyl group $W$ is isomorphic to the symmetric group $\operatorname{S}_3$. Elements of order three in $W$ are Coxeter elements, and  corresponding periodic components of the normalizer consist of elements of order three. Elements of order two in $W$ are reflections, and the corresponding components of the normalizer may contain elements of all even orders. For instance, for $w=(12)$, the component $N_w$ consists of matrices of the form
$$
\begin{pmatrix}
0 & t_1 & 0 \\
t_2 & 0 & 0 & \\
0 & 0 & -t_1^{-1}t_2^{-1}
\end{pmatrix}, \quad t_1,t_2\in\KK^{\times},
$$
and the set of elements of order $2s$ in this component is a disjoint union of curves of the form
$$
\begin{pmatrix}
0 & t & 0 \\
\epsilon t^{-1} & 0 & 0 & \\
0 & 0 & -\epsilon^{-1}
\end{pmatrix}, \quad t\in\KK^{\times},
$$
where $\epsilon$ is a primitive $s$-th root of unity in $\KK$. Note that if $s$ is divisible by the characteristic $p$ of $\KK$, the set of such primitive roots is empty.
\end{example}

\begin{remark}
The results of Theorem~\ref{tmain} are easily extended from the case of a semisimple group $G$ to the case of a reductive group $R$. If $G$ is a maximal semisimple subgroup in~$R$, then $R=(Z\times G)/H$, where $Z$ is a central subtorus in $R$ and $H$ is a finite central subgroup.
In this notation, the normalizer of the maximal torus $N_R((Z\times T)/H)$ coincides with $(Z\times N_G(T))/H$.

The Weyl group remains unchanged when we pass from $G$ to $R$. For elements of a given finite order in $(Z\times N_G(T))/H$, the first component consists of elements of bounded finite order in the torus $Z$, and the number of such elements is finite. In particular, $N_R((Z\times T)/H)$ has no periodic component in the case when $R$ is not semisimple.

If $d(w)=\ord(w)$, then the component of $N_R((Z\times T)/H)$ corresponding to the element $w$ contains elements of all orders of the form $d(w)s$, where $s$ runs through all positive integers coprime to the characteristic of the field $\KK$. If $d(w)=2\cdot\ord(w)$, then the orders of elements in the corresponding component may, depending on the choice of the subgroup $H$, range over both values of the form $\ord(w)s$ and values of the form $2\cdot\ord(w)s$. For example, the first possibility is realized for the group $\GL_2$, while the second is realized for $\KK^{\times}\times\SL_2$.
\end{remark}

In all cases, the set of elements of a given order is again a disjoint union of subvarieties $C_i$ of dimension $r-\sigma(w)$, but there may be more such subvarieties than in the semisimple case. For instance, if in Example~\ref{ex1} the group $\SL_3$ is replaced by $\GL_3$, then the subvariety of elements of order $2s$ in the component corresponding to $w=(12)$ has the form
$$
\begin{pmatrix}
0 & t & 0 \\
\epsilon_1 t^{-1} & 0 & 0 & \\
0 & 0 & \epsilon_2
\end{pmatrix}, \quad t\in\KK^{\times},
$$
where $\epsilon_1$ and $\epsilon_2$ are primitive $s_1$-th and $2s_2$-th roots of unity, respectively, such that the least common multiple of $s_1$ and $s_2$ is equal to $s$.

It is known that the number $\tau(w)$ of fixed points of the automorphism $\varphi_w$ of the torus $T$ is finite if and only if the element $w$ is elliptic~\cite{Fe}, see also Theorem~\ref{tmain}. It is easy to verify that if the elliptic elements $w$ and $w'$ are conjugate in $W$, then $\tau(w)=\tau(w')$.

\begin{proposition}\label{fix}
Let $w\in W$ be an elliptic element. Then the number $\tau(w)$ of fixed points of the automorphism $\varphi_w$ of the torus $T$ is equal to the multiplicity of the matrix $\Phi_w-E$.
\end{proposition}

It follows from Proposition~\ref{fix} that $\tau(w)$ does not depend on the choice of $G$ up to local isomorphism. In particular, this number remains unchanged when we pass from a simply connected group $G$ to an adjoint group $G$. Indeed, when passing to a finite covering of the torus at the level of character lattices, we obtain a superlattice of finite index. If we change the representation of the linear operator $\Phi_w-E$ by a matrix in a lattice basis to a matrix in a superlattice basis of finite index, then the determinant of the matrix and its $p'$-part will remain the same.

\begin{corollary} \label{nb}
If $w\in W$ is a Coxeter element, then $\tau(w)$ is equal to the determinant of the Cartan matrix $\mathcal{C}$ of the root system of $G$ if $\KK$ has characteristic zero, and to the $p'$-part of this determinant if $\KK$ has characteristic $p$.
\end{corollary}

\noindent Proofs of Proposition~\ref{fix} and Corollary~\ref{nb} are given in the next section.

\smallskip

If $Z$ is the center of the group $G$, then the elements of $Z$ lie in the set of fixed points of any inner automorphism. On the other hand, if $G$ is simply connected, the order of its center is exactly equal to $\det(\mathcal{C})$ in the case of a field of characteristic zero and to the $p'$-part of this number in the case of a field of characteristic~$p$. This shows that for a Coxeter element $w$ and a simply connected group $G$, the set of fixed points of the automorphism $\varphi_w$ coincides with the center $Z$. It follows that $\tau(w) \le \tau(w_1)$ for a Coxeter element~$w$ and any elliptic element $w_1$.

As an example, let us note that the fixed elements in the maximal torus of $\PSL_n=\SL_n/Z(\SL_n)$ with respect to the automorphism $\varphi_w$, where $w=(12\ldots n)$ is a Coxeter element, are precisely the classes $\text{diag}(1,\epsilon,\ldots,\epsilon^{n-1})Z(\SL_n)$, where $\epsilon$ runs through the $n$-th roots of unity.

\section{Proofs of main results}

Let us begin with several auxiliary observations. Each element $w\in W$ defines an automorphism $\varphi_w$ of the torus $T$ given by the formula $\varphi_w(t)=n_wtn_w^{-1}$.
The map
 $$
 \alpha_w\colon T\to T,  \quad t\mapsto \varphi_w^{-1}(t)t^{-1}
 $$
is a homomorphism of algebraic groups, the kernel of which is the subgroup of fixed points of the automorphism $\varphi_w$. Let $T(w)$ denote the image of the homomorphism $\alpha_w$. This is a closed subtorus of the torus $T$. In particular, $T(w)=\{1\}$ if and only if $w=e$, and $T(w)=T$ if and only if the automorphism $\varphi_w$ has a finite number of fixed points in the torus $T$.

\begin{lemma} \label{lem}
The orbits of the action of the torus $T$ on the component $N_w$ by conjugation are subsets of the form $n_wK$, where $K$ is a coset of $T$ with respect to the subtorus $T(w)$. In particular, for a fixed $w\in W$, all these orbits are closed and have the same dimension.
\end{lemma}

\begin{proof}
For any $t,t'\in T$, we have
$$
tn_wt't^{-1}=n_wn_w^{-1}tn_wt't^{-1}=n_w\varphi_w^{-1}(t)t't^{-1}=n_wt'\varphi_w^{-1}(t)t^{-1}.
$$
This shows that the $T$-orbit of the point $n_wt'$ with respect to the conjugation coincides with $n_wK$, where $K=t'T(w)$.
\end{proof}

\begin{proof}[Proof of Theorem~\ref{tmain}]
Let $d=\ord(w)$ and assume that $n_w$ is an element of the smallest order in the component $N_w$. Recall that $d(w)$ denotes the order of $n_w$ and that $d(w)=d$ or $2d$.

Note that for any $t\in T$, we have
$$
(n_wt)^d=n_w^d\beta_w(t), \quad \text{where} \quad \beta_w(t)=(n_w^{1-d}tn_w^{d-1})\ldots(n_w^{-1}tn_w)t=t\varphi_w(t)\ldots\varphi_w^{d-1}(t).
$$
The map $\beta_w\colon T\to T$ is a homomorphism. Let us fix a basis in the character lattice of the torus $T$, and let $\Phi_w$, $A_w$, and $B_w$ be the matrices of the homomorphisms $\varphi_w$, $\alpha_w$, and $\beta_w$ in this basis, respectively. Since the addition of matrices corresponds to the multiplication of homomorphisms, we have
$$
A_w=\Phi_w^{-1}-E \quad \text{and} \quad B_w=E+\Phi_w+\ldots+\Phi_w^{d-1}.
$$
It is easy to see that $B_wA_w=0$. This means that the image $T(w)$ of $\alpha_w$ lies in the kernel of~$\beta_w$. Furthermore, the rank of the matrix $A_w$ is $r-\sigma(w)$. Since the matrix $\Phi_w$ has order~$d$, this matrix is diagonalizable and its eigenvalues $\epsilon$ are $d$-th roots of unity. The eigenvalues of the matrix $B_w$ are of the form
$$
1+\epsilon+\ldots+\epsilon^{d-1}.
$$
This number is equal to $d$ if $\epsilon=1$, and is $0$ otherwise. This shows that the rank of $B_w$ is~$\sigma(w)$. Thus, the subtorus $T(w)$ is the identity component of the kernel of $\beta_w$, and the image of this homomorphism is a subtorus $S(w)$ of dimension $\sigma(w)$ in $T$. By the definition of $m(w)$, the kernel of $\beta_w$ is the union of $m(w)$ cosets $K_i$ of the torus $T$ by the subtorus~$T(w)$.

\smallskip

Consider two cases. First, let $d(w)=d$. If $d$ does not divide $k$, then $(n_wt)^k$ does not lie in $T$. Consequently, this element cannot be equal to the identity, and the set $D_k(w)$ is empty. Next, we assume $k=ds$. In this case, $(n_wt)^k=\beta_w(t)^s$. Thus, the element $n_wt$ lies in $D_k(w)$ if and only if $\beta_w(t)$ is an element of order $s$ in the subtorus $S(w)$. There are exactly $\nu(s,\sigma(w))$ such elements. Therefore, the set $D_k(w)$ is identified with the union of $m(w)\nu(s,\sigma(w))$ subvarieties $C_i=n_wK_i$, where $K_i$ are cosets of $T$ with respect to the subtorus $T(w)$. According to Lemma~\ref{lem}, these subvarieties coincide with the orbits of the action of $T$ on $N_w$ by conjugation.

\smallskip

Now let $d(w)=2d$. Then $t_0:=n_w^d$ is an element of order two in $T$. This element does not lie in the subtorus $S(w)$. Indeed, if $t_0=\beta_w(t)$, then
$$
(n_wt)^d=n_w^d\beta_w(t)=t_0t_0=e.
$$
This contradicts the condition $d(w)=2d$.

If $d$ does not divide $k$, then $(n_wt)^k$ does not lie in $T$ and the set $D_k(w)$ is empty. Next, we assume $k=dl$. Then $(n_wt)^k=t_0^l\beta_w(t)^l$. If $l$ is odd, the element $(n_wt)^k=t_0\beta_w(t)^l$ cannot be equal to the identity, and the set $D_k(w)$ is again empty. Therefore, we further assume that $k=d(w)s$, so $(n_wt)^k=\beta_w(t)^{2s}$.
It follows that $\beta_w(t)^{2s}=1$. If, moreover, the order of $n_wt$ is strictly less than $k$, then by what has been proven, this order is equal to $2ds'$ for some proper divisor $s'$ of $s$. In this case, $\beta_w(t)^{2s'}=1$. Consequently, if the order of $n_wt$ is $k$, the element $\beta(t)$ either has order $2s$ in $S(w)$ or has order $s''$, where the divisor $s''$ of $2s$ does not divide any proper even divisor of $2s$. The latter is possible only if $s$ is odd and $s''=s$.

Thus, in the case of even $s$, we find that the set $D_k(w)$ is the union of $m(w)\nu(2s,\sigma(w))$ subsets $C_i=n_wK_i$, where $K_i$ are cosets of $T$ with respect to $T(w)$. For odd $s$, the number of subsets $C_i=n_wK_i$ is ${m(w)(\nu(s,\sigma(w))+\nu(2s,\sigma(w)))}$: the first summand in this sum corresponds to the case $(n_wt)^{ds}=t_0$, and the second to the case $(n_wt)^{ds}=t_0s_0$, where $s_0$ is an element of order two in $S(w)$.

\smallskip

Theorem~\ref{tmain} is proved.
\end{proof}

\begin{proof}[Proof of Proposition~\ref{fod}]
Since a homomorphism maps elements of finite order to elements of finite order, for any element of finite order $t\in T$, the element $\beta_w(t)$, and thus the element $n_wt$, have finite order. Since elements of finite order are dense in the torus $T$, we conclude that elements of finite order are dense in the component $N_w=n_wT$.
\end{proof}

\begin{proof}[Proof of Proposition~\ref{fix}]
The condition $\varphi_w(t)=t$ is equivalent to the condition ${\varphi_w(t)t^{-1}=e}$, that is, $t$ lies in the kernel of the homomorphism $t\mapsto \varphi_w(t)t^{-1}$. The matrix of this homomorphism is $\Phi_w-E$. Since $w$ is elliptic, this matrix is nonsingular, and by definition, the multiplicity of this matrix is equal to the order of the kernel of the corresponding homomorphism.
\end{proof}

\begin{proof}[Proof of Corollary~\ref{nb}]
If $w\in W$ is a Coxeter element, then the map $1-w$ maps the weight lattice of the root system of $G$ onto the root lattice; see, for example, \cite[(VI)$\S$1, ex. 22a)]{Bur}. At the same time, the Cartan matrix is the transpose of the embedding matrix of the root lattice into the weight lattice, see \cite[(VI)$\S$1.10]{Bur}). This shows that
$$
|\det(\mathcal{C})|=|\det(\mathcal{C}^T)|=|\det(E-\Phi_w)|=|\det(\Phi_w-E)|.
$$
\end{proof}

%%%%%%%%%%%%%%%%%%%%%%%%%
\section{Tables}

In this section, we present the results obtained above, as well as well-known results, collected in two tables: for classical and exceptional simple linear groups.

Let us fix the notation used in the tables. A cyclic group of order $n$ is denoted by $\mathbb{Z}_n$ or simply $n$, and an abelian group of order $p^n$ with prime exponent $p$ is denoted by $p^n$. For two groups $A$ and $B$, the symbol $A.B$ denotes a group $G$ with a normal subgroup $A$ and a quotient group $B \simeq G/A$. The symbol $\operatorname{O_n}(q)$ denotes the simple orthogonal group of dimension $n$ over a field with $q$ elements.

The number of Coxeter elements of the Weyl group $W$ is denoted by $l_c$, and the order of a Coxeter element $w_c$ is denoted by $|w_c|$ for brevity. The latter is also called the Coxeter number. Its value is $|\Phi|/r$, where $|\Phi|$ is the total number of roots, see \cite[Proposition~3.18]{Hum}. The minimal order of a lift of an element $w_c$ to $N_G(T)$ is written as $|n_c|_{ad}$ or $|n_c|_{sc}$ for the case of an adjoint or simply connected group $G$, respectively. The number of fixed points of the automorphism $\varphi_{w_c}$ is denoted by $\tau(w_c)$.

For elliptic elements, the notation is similar, with the index $c$ replaced by $e$. Unlike Coxeter elements, there may be several conjugacy classes of elliptic elements in the Weyl group, and for brevity, we use the following conventions. The entry $\{d,2d\}$ in Table~\ref{Table_Classic} means that for different elliptic elements, the minimal order of a lift can be either $d$ or $2d$. For a specific elliptic element, the exact answer can be found in \cite{Za}. Note that the value $d$ may vary across different conjugacy classes of elliptic elements. In the case of exceptional groups, the exact answer is also given in \cite{Galt_Survey}, where the minimal orders of lifts for all elements of the Weyl group are found.

In Table~\ref{Table_Exceptional}, the row $|w_e|$ contains the orders of representatives of the conjugacy classes of elliptic elements. In the case of several conjugacy classes of elements of a given order, the subscripts indicate their number. The indices in rows $|n_e|_{ad}$, $|n_e|_{sc}$, and $\tau(w_e)$ are defined similarly. Note that the conjugacy class corresponding to Coxeter elements does not appear in any of these rows.

For example, in the $\operatorname{E_8}$ column of Table~\ref{Table_Exceptional}, the entry $|w_e|=4_2$ means that there are two conjugacy classes of elliptic elements of order 4, and the entry $\tau(w_e)=\{16,64\}$ below means that for one class, the number of fixed points of the automorphism $\varphi_{w_e}$ is 16, and for the other, it is 64. The symbol $+$ means that the minimal order of the preimage is equal to the order of the corresponding element of the Weyl group.

At the end of the tables, the derived subgroup $W'$ of the Weyl group and the subgroup $\operatorname{C}$ generated by all Coxeter elements are specified. For the subgroup $\operatorname{E}$ generated by all elliptic elements of the Weyl group, direct calculations show that $\operatorname{C}=\operatorname{E}$, except for the root system $\operatorname{F_4}$. In the latter case, $|W:\operatorname{E}|=2$ and $|\operatorname{E}:\operatorname{C}|=2$.

\begin{remark}
The structure of the group $N_{\operatorname{C}}$ generated by all periodic components follows directly from the tables, since $N_{\operatorname{C}}=\langle T, \operatorname{C}\rangle$. As noted in Remark~\ref{rem2}, the derived subgroup $N'$ of the normalizer $N_G(T)$ contains every element of the maximal torus $T$; therefore, $N'$ consists of the components corresponding to the elements of $W'$. In particular, it follows from the tables that $N'\leq N_{\operatorname{C}}\leq N_G(T)$, and both inclusions are strict only for type~$\operatorname{G}_2$.
\end{remark}

According to Corollary~\ref{nb}, for a Coxeter element $w$, the value $\tau(w)$ coincides with the order of the center of the simply connected group, which is
well-known; see, e.g.,~\cite[Table~3]{VinbergOn}.

Now we explain how the values $\tau(w_e)$ are found for an arbitrary elliptic element $w_e$. The parametrization of conjugacy classes of the Weyl group ${W}$ of classical type is given in~\cite[\S7]{CarWeyl}. All elliptic elements of the group ${W(\operatorname{A_n})}$ are Coxeter elements. In the group ${W(\operatorname{B_n})}\simeq {W(\operatorname{C_n})}$, the conjugacy classes of elliptic elements are in one-to-one correspondence with the partitions of the number~$n$. The group ${W(\operatorname{D_n})}$ is a subgroup of index 2 in the group ${W(\operatorname{B_n})}$, and its conjugacy classes of elliptic elements are in one-to-one correspondence with the partitions of $n$ into an even number of parts.

\begin{proposition}
Let $G$ be a simple linear algebraic group with root system $\Phi\in\{\operatorname{B_n}, \operatorname{C_n}, \operatorname{D_n}\}$. If $w\in W$ is an elliptic element corresponding to the partition $n=n_1+n_2+\ldots+n_h$ ($h$ is even in the case $\Phi=\operatorname{D_n}$), then $\tau(w)=2^h$.
\end{proposition}

\begin{proof}
It follows from~\cite[Proposition~24]{CarWeyl} that an elliptic element $w$ has an admissible diagram of the form $\operatorname{B_{n_1}}\times\operatorname{B_{n_2}}\times\ldots\times\operatorname{B_{n_h}}$. Consequently, there exists a basis $$\{e_{11},\ldots,e_{1n_1}, e_{21},\ldots,e_{2n_2},\ldots,e_{r1},\ldots,e_{hn_h}\},$$ consisting of mutually orthogonal sets of vectors, i.e., $(e_{ij},e_{kl})=0$ for all $i\neq k$. In this basis, the matrix $A_w$ of $w$ has a block-diagonal form, where each block corresponds to an element with a diagram of type $\operatorname{B_{n_i}}$. An element with a diagram of type $\operatorname{B_{n_i}}$ corresponds to a Coxeter element $w_i$ for the root subsystem $\operatorname{B_{n_i}}$. In the case $n_i=1$, we identify the root systems $B_1$ and $A_1$, where $\tau(w_i)=2$. Thus,
$$\tau(w)=|\operatorname{det}(A_w-E)|=|\operatorname{det}(A_{w_1}-E_1)\ldots\operatorname{det}(A_{w_h}-E_h)|=2^h.$$
\end{proof}

For the exceptional groups, since the order of the corresponding Weyl group is bounded, the values $\tau(w)$ for elliptic elements $w$ were obtained by direct calculations according to Proposition~\ref{fix}.

\begin{table}[H]
	\centering
	\caption{Weyl groups for classical simple groups}\label{Table_Classic}
	\begin{tabular}{|c|c|c|c|c|}
		\hline
		& $\operatorname{A_{n-1}}$ & $\operatorname{B_n}, \, n\geq2$ & $\operatorname{C_n}, \, n\geq2$ & $\operatorname{D_n}, \, n\geq4$ \\ \hline
		${W}$ & $\operatorname{S_n}$ & $\operatorname{2^n\rtimes S_n}$ & $\operatorname{2^n\rtimes S_n}$ & $\operatorname{2^{n-1}\rtimes S_n}$  \\ \hline
		$|{W}|$ & $n!$ & $2^nn!$ & $2^nn!$ & $2^{n-1}n!$  \\ \hline
		$l_c$ & $(n-1)!$ & $2^{n-1}(n-1)!$ & $2^{n-1}(n-1)!$ & $2^{n-2}(n-2)!n$  \\  \hline
		$|w_c|$ & $n$ & $2n$ & $2n$ & $2(n-1)$  \\  \hline
		$|n_c|_{ad}$ & $n$ & $2n$ & $2n$ & $2(n-1)$  \\  \hline
		$|n_c|_{sc}$ & $2n$,  \text{$n$ even} 	& $2n$,  $n\equiv0,3\!\pmod4$ & $4n$ & $2(n-1)$,  $n\equiv0,1\!\pmod4$  \\
		& $n$,  \text{$n$ odd} & $4n$,  $n\equiv1,2\!\pmod4$ &  & $4(n-1)$,  $n\equiv2,3\!\pmod4$  \\  \hline
		$\tau(w_c)$ & $n$ & $2$ & $2$ & $4$  \\  \hline
		$l_e$ & $(n-1)!$ & $(2n-1)!!$ & $(2n-1)!!$ & $(2n-3)!!(n-1)$  \\  \hline
		$|w_e|$ & $-$ & $d$ & $d$ & $d$  \\  \hline
		$|n_e|_{ad}$ & $-$ & $d$ & $\{d, 2d\}$ & $d$  \\  \hline
		$|n_e|_{sc}$ & $-$ & $\{d, 2d\}$ & $2d$ & $\{d, 2d\}$  \\  \hline
		$\tau(w_e)$ & $-$ & $2^h$ & $2^h$ & $2^h$  \\  \hline
		${W}'$ & $\operatorname{A_n}$ & $\operatorname{2^{n-1}\rtimes A_n}$ & $\operatorname{2^{n-1}\rtimes A_n}$ & $\operatorname{2^{n-1}\rtimes A_n}$  \\  \hline
		$\operatorname{C}$ & ${W},\,  \text{$n$ even}$ 	& ${W},\,  \text{$n$ even}$  & ${W},\, \text{$n$ even}$  & ${W},\,  \text{$n$ even}$   \\
		& ${W}',\,  \text{$n$ odd}$ & ${W}',\,  \text{$n$ odd}$  & ${W}',\,  \text{$n$ odd}$  & ${W}',\,  \text{$n$ odd}$   \\  \hline
	\end{tabular}
\end{table}

\begin{table}[H]
	\centering
	\caption{Weyl groups for exceptional simple groups}\label{Table_Exceptional}
	\begin{tabular}{|c|c|c|c|c|c|}
		\hline
		& $\operatorname{G_2}$ & $\operatorname{F_4}$ & $\operatorname{E_6}$ & $\operatorname{E_7}$ & $\operatorname{E_8}$ \\ \hline
		${W}$ & $\operatorname{D_{6}}$ & Solvable & $\operatorname{O_5(3)\rtimes2}$ & $\operatorname{O_7(2)\times2}$  & $(2.\operatorname{O_8^+(2)})\rtimes2$ \\ \hline
		$|{W}|$ & $12$ & $2^7\!\cdot\! 3^2$ & $2^7\!\cdot3^4\!\cdot5$ & $2^{10}\!\cdot3^4\!\cdot5\!\cdot7$  & $2^{14}\!\cdot3^5\!\cdot5^2\!\cdot7$ \\ \hline
		$l_c$ & $2$ & $96$ & $2^5\!\cdot3^3\!\cdot5$ & $2^9\!\cdot3^2\!\cdot5\!\cdot7$  & $2^{13}\!\cdot3^4\!\cdot5\!\cdot7$ \\  \hline
		$|w_c|$ & $6$ & $12$ & $12$ & $18$  & $30$ \\  \hline
		$|n_c|_{ad}$ & $6$ & $12$ & $12$ & $18$ & $30$ \\  \hline
		$|n_c|_{sc}$ & $6$ & $12$ & $12$ & $36$ & $30$ \\  \hline
		$\tau(w_c)$ & $1$ & $1$ & $3$ & $2$  & $1$ \\  \hline
		$l_e$ & $5$ & $385$ & $2^5\!\cdot5\!\cdot7\!\cdot11$ & $3^2\!\cdot5\!\cdot7\!\cdot11\!\cdot13\!\cdot17$  & $7\!\cdot11\!\cdot13\!\cdot17\!\cdot19\!\cdot23\!\cdot29$  \\  \hline
		$|w_e|$ & $2,3_2$ & $2,3,4_2,6_3,8$ & $3,6_2,9$ & $2,4,6_4,8$  & $2,3,4_2,5,6_6$ \\
			&  &  &  & $10,12,14,30$  & $8_2,9,10_2,12_6,14$ \\
			&  &  &  &   & $15,18_2,20,24,30$ \\  \hline
		$|n_e|_{ad}$ & $+$ & $2,3,8_2,6_3,8$ & $+$ & $+$ & $+$ \\  \hline
		$|n_e|_{sc}$ & $+$ & $2,3,8_2,6_3,8$ & $+$ & $4,4,12_4,8$ & $+$ \\
		    &  &  &  & $20,12,28,60$ &  \\  \hline
		$\tau(w_e)$ & $1,1_2$ & $16,9,\{4,8\}$ & $27,\{3,12\},3$ & $128,32,$  & $256,81,\{16,64\},25,$ \\
	    	&  & $\{1,4_2\},2$ &  & $\{2,8,18,32\}$  & $\{1,4,9,16,36,64\},\{4,16\},$ \\
	    	&  &  &  & $8,8,2,2,2$ & $9,\{1,16\},\{1_2,4_2,9,16\},$ \\
			&  &  &  &  & $4,1,\{1,4\},1,1,1$ \\  \hline
		${W}'$ & $\mathbb{Z}_3$ & $|{W}:{W}'|=4$ & $\operatorname{O_5(3)}$ & $\operatorname{O_7(2)}$  & $2.\operatorname{O_8^+(2)}$ \\ \hline
		$\operatorname{C}$ & $\mathbb{Z}_6$ & ${W}'$ & $W'$ & ${W}$  & ${W}'$ \\ \hline
		\end{tabular}
	\end{table}

\end{document}